\documentclass[12pt]{amsart}

\usepackage{amssymb}

\usepackage{amsthm}
\usepackage{mathtools}
\usepackage{latexsym}
\usepackage{amsmath}
\usepackage{amssymb}

\usepackage{amsthm}
\usepackage[all]{xy}

\usepackage{hyperref}

\hypersetup{
    colorlinks=true,       
    linkcolor=blue,          
    citecolor=blue,        
    filecolor=blue,      
    urlcolor=blue           
}

\newtheorem{thm}{Theorem}
\newtheorem{prop}[thm]{Proposition}

\newtheorem{lem}[thm]{Lemma}

\newtheorem{rem}[thm]{Remark}

\newtheorem{ex}[thm]{Example}

\begin{document}

\title[On a Homma-Kim conjecture]{On a Homma-Kim conjecture for nonsingular hypersurfaces}

\author{Andrea Luigi Tironi}

\date{\today}

\address{
Departamento de Matem\'atica,
Universidad de Concepci\'on,
Casilla 160-C,
Concepci\'on, Chile}
\email{atironi@udec.cl}

%\thanks{During the preparation of this paper, the author was partially supported 
%by Proyecto VRID N. 214.013.039-1.OIN and the Project Anillo ACT 1415 PIA CONICYT}

\maketitle

\begin{abstract}
Let $X^n$ be a nonsingular hypersurface of degree $d\geq 2$ in the projective space $\mathbb{P}^{n+1}$ 
defined over a finite field $\mathbb{F}_q$ of $q$ elements.  
We prove a Homma-Kim conjecture on a upper bound about the number of $\mathbb{F}_q$-points of $X^n$ for $n=3$, and for any odd integer $n\geq 5$ and $d\leq q$. 
\end{abstract}

\section{Introduction}\label{intro}

Let $\mathbb{F}_q$ be a finite field with $q$ elements and consider a hypersurface $X^n$ in the projective space $\mathbb{P}^{n+1}$ defined over $\mathbb{F}_q$ of degree $d\geq 2$ and dimension $n\geq 3$.

There is a vast literature on the general problem of counting or finding bounds, especially upper bounds, on the number of $\mathbb{F}_q$-points of certain projective varieties and hypersurfaces defined over finite fields $\mathbb{F}_q$
(see, e.g., \cite{LW}, \cite{C}, \cite{RS}, \cite{BSS} and the references therein). 

Recently, Homma and Kim  established the following upper bound involving a Thas' invariant $k_{X^n}$ (\cite[Theorem 3.2]{HK7}, \cite[Proposition 3]{T1}),
\begin{equation}\label{k-bound}
N_q(X^n)\leq (d-1)q^{k_{X^n}}N_q(\mathbb{P}^{n-k_{X^n}})+N_q(\mathbb{P}^{k_{X^n}})
\end{equation}
which is sharp for $k_{X^n}>0$ (see \cite[Proposition $5$]{T1} for $k_{X^n}=0$), where $k_{X^n}$ is the dimension of a maximal $\mathbb{F}_q$-linear subspace $\mathbb{P}^{k_{X^n}}$ contained in $X^n$. 
Moreover, in \cite{T2} and \cite{T1} equalities in \eqref{k-bound} for any $k_{X^n}\leq n-1$ were considered and completely characterized by giving, up to projectivities, the complete list of all the hypersurfaces which reach the bound in \eqref{k-bound}.

On the other hand, in \cite{HK7} 
Homma and Kim observed that the hypersurfaces appearing in \cite{T2} are all singular, except when $n=2$ or $d=q+1$, concluding that if we restrict the investigation within nonsingular hypersurfaces, one can expect a tighter bound than the elementary bound in \cite{HK4}, that is, the bound \eqref{k-bound} with $k_{X^n}=n-1$. In line with this idea, they proved in \cite{HK7} that $k_{X^n}\leq \lfloor \frac{n}{2}\rfloor$ for nonsingular hypersurfaces $X^n$ and they showed
that nonsingular hypersurfaces $X^n$ in $\mathbb{P}^{n+1}$ with $n\geq 2$ even, really reach the equality in
\eqref{k-bound} with $k_{X^n}=\lfloor\frac{n}{2}\rfloor=\frac{n}{2}$ by giving, up to projectivities, a complete list of such hypersurfaces (see \cite[Theorem 4.1]{HK7}).

As to the case $n$ odd, they observed that equality in
\eqref{k-bound} with $k_{X^n}=\lfloor\frac{n}{2}\rfloor=\frac{n-1}{2}$ does not occur in the nonsingular case and
they finally proposed the following conjecture (see \cite[$\S 7.2$]{HK7}): 

\medskip

\noindent \textit{Conjecture (H-K)}. Let $X^n$ be a nonsingular hypersurface of degree $d\geq 2$ in $\mathbb{P}^{n+1}$ over $\mathbb{F}_q$. If $n\geq 3$ is an odd integer, then 
\begin{equation}\label{conj}
N_q(X^n)\leq\theta_n^{d,q}:=(d-1)q^{\frac{n+1}{2}}N_q(\mathbb{P}^{\frac{n-1}{2}})+N_q(\mathbb{P}^{\frac{n-1}{2}})\ . \tag{*}
\end{equation}

\medskip

\noindent Let us observe here that 
the upper bound \eqref{conj} is not trivial for $d\leq q+1$ and that in these cases it is sharp because there are at least three examples with $d=2,\sqrt{q}+1,q+1$, which satisfy the equality in \eqref{conj} (see e.g. \cite[$\S 7.2$]{HK7} for $d\leq q$ and Example \ref{example} in Section \ref{sec2} for $d=q+1$). 

\medskip

The purpose of this note is to show that \textit{Conjecture (H-K)} is true in the case $n=3$ (see \cite{D} for $n=3, d\leq q$ and $(d,q)\neq (4,4)$) and for any odd integer $n\geq 5$ when $d\leq q$. 

\smallskip

More precisely, we obtain the following two main results.

\begin{thm}\label{thm1}
Let $X^3$ be a nonsingular hypersurface of degree $d\geq 2$ in $\mathbb{P}^{4}$ defined over $\mathbb{F}_q$. Then 
$$N_q(X^3)\leq\theta_3^{d,q}\ ,$$
and equality holds only if there exists an $\mathbb{F}_q$-point $P\in X^n$ such that $X^3\cap T_PX^3$ is a cone $P*Y$ with vertex $P$ over a nonsingular curve $Y\subset\mathbb{P}^{2}$ such that $N_q(Y)=(d-1)q+1$.
\end{thm}

\begin{thm}\label{thm2}
Let $X^n$ be a nonsingular hypersurface of degree $d\geq 2$ in $\mathbb{P}^{n+1}$ defined over $\mathbb{F}_q$ with $n\geq 5$ an odd integer. If $d\leq q$, then 
$$N_q(X^n)\leq\theta_n^{d,q}\ ;$$
moreover, the equality is reached by a nonsingular hypersurfaces $X^n$ in $\mathbb{P}^{n+1}$ only if there exists an $\mathbb{F}_q$-point $P\in X^n$ such that $X^n\cap T_PX^n$ is a cone $P*Y$ with vertex $P$ over a nonsingular hypersurface $Y\subset\mathbb{P}^{n-1}$ such that $N_q(Y)=\theta_{n-2}^{d,q}$.
\end{thm}

\smallskip

\noindent\textbf{Acknowledgment}. The author 
wishes to thank 
M. Datta for having drawn his attention
in April 2019 to the preprint in arXiv of \cite{D} which partially inspired and motivated 
this work. During the preparation of this paper, the author was partially supported by the Project VRID N. 219.015.023-INV.

\section{Notation and preliminary results}\label{sec1}

Let $X^n$ be a hypersurface in $\mathbb{P}^{n+1}$ of degree $d\geq
2$ and dimension $n\geq 3$ defined over a finite field
$\mathbb{F}_q$ of $q$ elements, where $q=p^r$ for some prime number
$p$ and an integer $r\in\mathbb{Z}_{\geq 1}$. If $W$ is an algebraic set in $\mathbb{P}^{n+1}$
defined by some polynomials $F_1,\dots,F_s\in\mathbb{F}_q[x_0,\dots,x_{n+1}]$, we write $W=V(F_1,\dots,F_s)$ and we denote by $W(\mathbb{F}_q)$ and $N_q(W)$
the set of $\mathbb{F}_q$-points of $W$ 
and the cardinality of $W(\mathbb{F}_q)$, 
respectively. Moreover, we denote by $k_W$ the Thas' invariant of $W$, i.e. the maximal dimension of an $\mathbb{F}_q$-linear subspace contained in $W$ (\cite{Th} and \cite[Definition 3.1]{HK7}).

Recall that for any $N\in\mathbb{Z}_{\geq 0}$ we have
$$N_q(\mathbb{P}^{N})=q^N+q^{N-1}+\dots +q+1\ .$$

Moreover, let us denote here by $V*Y$ the cone with vertex $V\cong\mathbb{P}^h$ for some $h\in\mathbb{Z}_{\geq 0}$ over an algebraic set $Y$ defined over $\mathbb{F}_q$ and by $\langle Z,W\rangle$ the smaller $\mathbb{F}_q$-linear subspace containing the algebraic sets $Z$ and $W$ defined over $\mathbb{F}_q$.
Finally, the nonsingularity of $X^n$ is normally 
(and also in this note) considered over the algebraic closure $\overline{\mathbb{F}}_q$ of $\mathbb{F}_q$.

\begin{rem}\label{rem2}
Let $X^n$ be a nonsingular hypersurface of degree $d\geq 2$ in 
$\mathbb{P}^{n+1}$ with $n\geq 3$ defined over $\mathbb{F}_q$. Then $k_{X^n}\leq\lfloor\frac{n}{2}\rfloor$ $($\cite[Lemma 2.1]{HK7}$)$. 
\end{rem}

Let us give here some preliminary results
which will be useful 
to prove Theorems \ref{thm1} and \ref{thm2} of the Introduction.

\begin{lem}\label{lem-cone}
Let $X^n$ be a nonsingular hypersurface of degree $d\geq2$ in $\mathbb{P}^{n+1}$ with $n\geq 3$ defined over $\mathbb{F}_q$. Suppose that there exists a point $P\in X^n(\mathbb{F}_q)$ such that 
$T_PX^n\cap X^n$ is a cone $P*Y$ over some hypersurface $Y$ in $\mathbb{P}^{n-1}$. 
Then $Y$ is nonsingular. 
\end{lem}

\begin{proof}
Assume that
$Y$ contains a singular point $Q$. After a projectivity, we can write
$$P:=(1:0:\dots :0),\ T_PX^n:= \{ x_{n+1}=0 \},\ Q:=(0:1:0:\dots :0)\ ,$$
$X^n:=V(F)$ for some homogeneous polynomial $F=x_{n+1}G+H\in\overline{\mathbb{F}}_q[x_0,\dots,x_{n+1}]$,
with $G\in\overline{\mathbb{F}}_q[x_0,\dots,x_{n+1}]$, $H\in \overline{\mathbb{F}}_q[x_1,\dots,x_{n}]$ such that $G(1,0,\dots,0)\neq 0, H(1,0,\dots,0)=0$, and $Y=V(x_0,x_{n+1},H)$.  
Since $Q$ is a singular point of $Y$, we have $\frac{\partial H}{\partial x_i}(t,0,\dots,0)=0$ for any $t\in\overline{\mathbb{F}}_q$ and for
every $i=1,\dots, n$. Moreover, note that $P_{a,b}:=(a:b:0:\dots:0)\in X^n$ for any $(a:b)\in\mathbb{P}^1(\overline{\mathbb{F}}_q)$ and 
$$\frac{\partial F}{\partial x_0}=x_{n+1}\frac{\partial G}{\partial x_0}, \ \frac{\partial F}{\partial x_i}=x_{n+1}\frac{\partial G}{\partial x_i}+\frac{\partial H}{\partial x_i}, \ \frac{\partial F}{\partial x_{n+1}}=G+x_{n+1}\frac{\partial G}{\partial x_{n+1}}\ ,$$
for every $i=1,\dots,n$. Therefore,  
we deduce that $\frac{\partial F}{\partial x_j}(P_{a,b})=0$ for $j=0,\dots,n$, and $\frac{\partial F}{\partial x_{n+1}}(P_{a,b})=G(a,b,0,\dots,0)$. Since $G(1,0,\dots,0)\neq 0$, we see that $$g(t):=G(1,t,0,\dots,0)\in\overline{\mathbb{F}}_q[t]$$ is not the null polynomial. Suppose that $g(t)$ is not a constant. Then there exists $b_0\in\overline{\mathbb{F}}_q$ such that $g(b_0)=0$. Thus $P_{1,b_0}$ is a singular point of $X^n$, which is a contradiction. Therefore, $g(t)$ is a non-zero constant polynomial, that is, $g(t)=c\in\overline{\mathbb{F}}_q\setminus\{0\}$ for every $t\in\overline{\mathbb{F}}_q$. 

Thus we deduce that $$G(x_0,x_1,\dots,x_{n+1})=cx_0^{d-1}+\sum_{i=2}^{n+1}x_iG_i(x_0,\dots,x_{n+1}),$$ where $\deg G_i=d-2$. Hence 
$\frac{\partial F}{\partial x_j}(P_{0,1})=0$ for $j=0,\dots,n$, and $\frac{\partial F}{\partial x_{n+1}}(P_{0,1})=0$, i.e. $P_{0,1}$ is a singular point of $X^n$, but this gives again a contradiction.
\end{proof}

\begin{lem}\label{lem-irred}
Let $X^n$ be a nonsingular hypersurface of degree $d\geq 2$ in 
$\mathbb{P}^{n+1}$ defined over $\mathbb{F}_q$ and let $P$ be an $\mathbb{F}_q$-point of $X^n$. If $n\geq 3$ is an odd integer, then either $T_PX^n\cap X^n$ is irreducible over $\mathbb{F}_q$, or $N_q(X^n)\leq\theta_n^{d,q}$, with equality only if $n=3$ and $d=q+1$. 
\end{lem}

\begin{proof}
Note that if $d\geq q+2$, then $N_q(X^n)<\theta_n^{d,q}$. So, we can assume that $d\leq q+1$ and $T_PX^n\cap X^n=V_1\cup V_2$ is reducible over $\mathbb{F}_q$. Since $k_{V_j}\leq k_{X^n}$ and $d_j:=\deg V_j\leq q+1$ for $j=1,2$, by \eqref{k-bound} we have
$$N_q(V_j)\leq (d_j-1)q^{k_{X^n}}N_q(\mathbb{P}^{n-1-k_{X^n}})+N_q(\mathbb{P}^{k_{X^n}})=:\delta_j^{(n)}\ .$$

Suppose $n=3$ and, after renaming, assume that $P\in V_1$. Observe that $N_q(V_j)=\delta_j^{(3)}$ implies that $V_j$ is a nonsingular surface in $\mathbb{P}^3$ by \cite{HK5, T2} and that if $N_q(V_j)=\delta_j^{(3)}$ for $j=1,2$, then $P\in V_1\cap V_2$ because $P$ is a singular point of $V_1\cup V_2$. So, in any case, we get
$$N_q(X^3)\leq (d-1)q^3+\delta_1^{(3)}+\delta_2^{(3)}-1\ ,$$
and since $k_{X^3}\leq 1$ by Remark \ref{rem2} and $d=d_1+d_2$, we conclude that
$$N_q(X^3)\leq (d-1)q^3+(d-1)qN_q(\mathbb{P}^1)-qN_q(\mathbb{P}^1)+2N_q(\mathbb{P}^1)-1=$$
$$=(d-1)q^3+(d-1)q^2+(d-1)q-q(q+1)+2(q+1)-1=$$
$$=\theta_3^{d,q}+(d-1)q-q^2\leq \theta_3^{d,q}\ ,$$
with equality only if $d=q+1$.

Assume now $n\geq 5$. Since $k_{X^n}\leq \frac{n-1}{2}$ by Remark \ref{rem2}, we obtain that
$$N_q(X^n)\leq (d-1)q^n+N_q(V_1)+N_q(V_2)\leq (d-1)q^n+\delta_1^{(n)}+\delta_2^{(n)}\leq$$
$$\leq (d-1)q^n+dq^{\frac{n-1}{2}}N_q(\mathbb{P}^{\frac{n-1}{2}})-2q^{\frac{n-1}{2}}N_q(\mathbb{P}^{\frac{n-1}{2}})+2N_q(\mathbb{P}^{\frac{n-1}{2}})=$$
$$=(d-1)q^n+(d-1)(q^{n-1}+\dots+q^{\frac{n-1}{2}})
-(q^{n-1}+\dots+q^{\frac{n-1}{2}})+2N_q(\mathbb{P}^{\frac{n-1}{2}})=$$
$$=\theta_n^{d,q}+(d-1)q^{\frac{n-1}{2}}-(q^{n-1}+\dots+q^{\frac{n-1}{2}})+N_q(\mathbb{P}^{\frac{n-1}{2}})\leq$$
$$\leq \theta_n^{d,q}+q^{\frac{n+1}{2}}+N_q(\mathbb{P}^{\frac{n-1}{2}})-(q^{n-1}+\dots+q^{\frac{n-1}{2}})=$$
$$=\theta_n^{d,q}+N_q(\mathbb{P}^{\frac{n-3}{2}})-(q^{n-1}+\dots+q^{\frac{n+3}{2}})<\theta_n^{d,q}\ ,$$ because $n-1\geq\frac{n+3}{2}>\frac{n-3}{2}$ for $n\geq 5$.
\end{proof}

\section{Proof of Theorem \ref{thm1}}\label{sec2} 

In this Section, we will first consider the cases $d\leq q$ (Subsection \ref{sub1}) and then we will treat the remaining cases $d\geq q+1$ (Subsection \ref{sub2}). 

Let us recall that the case $n=3$ was already considered by
M. Datta in \cite{D} under the assumption that $d\leq q$ and
$(d,q)\neq (4,4)$. So, for this reason and for the convenience of the reader,
to prove Theorem \ref{thm1} we will first revisit completely the cases $d\leq q$.

\medskip

First of all, let us prove a preliminary result which in fact holds for any $d\geq 2$.

\begin{prop}\label{n=3 cone}
Let $X^3$ be a nonsingular hypersurface of degree $d\geq 2$ in $\mathbb{P}^{4}$ defined over $\mathbb{F}_q$. Suppose that there exists a point $P\in X^3(\mathbb{F}_q)$ such that 
$T_pX^3\cap X^3$ is a cone $P*\gamma$ over a plane curve $\gamma$. Then $$N_q(X^3)\leq\theta_3^{d,q}\ ,$$
and the equality is attained only if $d\leq q+1$ and $\gamma$ is a nonsingular plane curve defined over $\mathbb{F}_q$ with $N_q(\gamma)=(d-1)q+1$.
\end{prop}

\begin{proof}
By Lemma \ref{lem-cone} we know that $\gamma$ is a nonsingular plane curve defined over $\mathbb{F}_q$, hence without $\mathbb{F}_q$-lines. Then from \cite{HK3} it follows that either 

\smallskip

(i) $N_q(\gamma)\leq (d-1)q+1$, or 

\smallskip

(ii) $N_q(\gamma)=(d-1)q+2$ with $(d,q)=(4,4)$. 

\smallskip

\noindent In case (i), we get 
$$N_q(X^3)\leq (d-1)q^3+qN_q(\gamma)+1\leq (d-1)q^3+(d-1)q^2+q+1\ ,$$ that is $N_q(X^3)\leq\theta_3^{d,q}$, and if equality holds, then $N_q(\gamma)=(d-1)q+1$. 

\smallskip

Now, suppose we are in case (ii) with $d=q=4$. 
Consider an $\mathbb{F}_4$-line $l=\langle P,A\rangle$ passing through $P$ and an $\mathbb{F}_4$-point $A\in\gamma$.

\bigskip

\noindent\textit{Claim $1$}. $T_QX^3\cap X^3$ is not a cone for every  $Q\in l(\mathbb{F}_4)\setminus\{P\}$.

\medskip

\noindent\textit{Proof.} Suppose there exists an $\mathbb{F}_4$-point $Q\in l$ such that $Q\neq P$ and $T_QX^3\cap X^3$ is a cone $Q*\gamma'$ over a (nonsingular) plane curve $\gamma'$. Note that $T_QX^3\neq T_PX^3$, otherwise $T_QX^3\cap X^3=T_PX^3\cap X^3$ would not be singular at $Q$. 
Define $\mathbb{P}_{QP}^2:=T_Q X^3\cap T_P X^3$ and observe that $l\subset\mathbb{P}_{QP}^2$. Let $\mathbb{P}_{P}^2$ be the $\mathbb{F}_4$-plane containing $\gamma$ and consider the $\mathbb{F}_4$-line $l':=\mathbb{P}_{QP}^2\cap \mathbb{P}_{P}^2$. Then $A\in l'\cap \gamma$, i.e. $l'$ is a secant $\mathbb{F}_4$-line of $\gamma$. Since $N_q(\gamma)=(d-1)q+2=14$, we deduce that there must be another $\mathbb{F}_4$-point $B\in l'\cap \gamma\subset T_Q X^3\cap X^3$ distinct from $A$ and $Q$. Thus there is an $\mathbb{F}_4$-line $l''$ passing through $Q$ and $B$ such that $l''\subset T_P X^3\cap X^3=P*\gamma$. Since $d=q$, this shows that $P*\gamma$ contains an $\mathbb{F}_4$-plane passing through $P$ and $l''$, which is impossible because $X^3$ is nonsingular. \hfill{Q.E.D.}  

\bigskip

By \textit{Claim} $1$, we know that $T_Q X^3\cap X^3$ is not a cone for every $Q\in l(\mathbb{F}_4)$ such that $Q\neq P$. Take an $\mathbb{F}_4$-plane $\Pi\nsubseteq T_P X^3$ such that $l\cap\Pi=\emptyset$ and $\Pi\cap\mathbb{P}_P^2=L$ is an $\mathbb{F}_4$-line such that $A\notin L$, where $\mathbb{P}_P^2$ is the $\mathbb{F}_4$-plane containing $\gamma$. Denote by $\mathbb{P}_{x,l}^2$ the $\mathbb{F}_4$-plane passing through $x\in\Pi(\mathbb{F}_4)$ and containing $l$. Then
$$N_4(X^3)\leq N_4(P*\gamma)+\sum_{x\in (\Pi\setminus L)(\mathbb{F}_4)}N_4(\mathbb{P}_{x,l}^2\cap X^3)\ .$$
Note that $\mathbb{P}_{x,l}^2\cap X^3$ does not contain $\mathbb{F}_4$-lines distinct from $l$ and passing through $P$ for every $x\in (\Pi\setminus L)(\mathbb{F}_4)$. Moreover, observe that $N_4(\Pi\setminus L)=4^2=16$. So, if $\Omega_l(Q)$ is the set of all $\mathbb{F}_q$-planes containing $l$ and such that their intersection with $X^3$ is a union of $d$ distinct lines defined over $\mathbb{F}_q$ and passing through $Q\in l(\mathbb{F}_q)$ (here $d=q=4$), by \textit{Claim} $1$ and \cite[Theorem 3.1]{D} we see that $|\Omega_l(Q)|\leq \frac{d(d-1)-1}{d-1}=\frac{11}{3}$, i.e. $|\Omega_l(Q)|\leq 3$ for every $Q\in l(\mathbb{F}_4)\setminus\{P\}$. Define 
$$\overline{\Omega}(l):=\bigcup_{Q\in l(\mathbb{F}_4)\setminus\{P\}}\Omega_l(Q)\ ,\quad \overline{\mathcal{B}}(l):=\left\{\mathbb{P}_{x,l}^2\ :\ x\in (\Pi\setminus L)(\mathbb{F}_4)\right\} ,$$
$$\overline{\sigma}(l):=\left\{ \mathbb{P}_{x,l}^2\in\overline{\mathcal{B}}(l)\ |\ \mathbb{P}_{x,l}^2\cap X^3\ \mathrm{is\ a\ union\ of\ } 4\ \mathrm{lines\ defined\ over\ } \mathbb{F}_4 \right\} .$$
Recalling that if $\mathbb{P}_{x,l}^2\in \overline{\sigma}(l)$, then $\mathbb{P}_{x,l}^2\cap X^3$ does not contain an $\mathbb{F}_4$-line distinct from $l$ and passing through $P$, we get $|\overline{\Omega}(l)|\leq (d-1)q=12$ and $$|\overline{\sigma}(l)|\leq\frac{[(d-1)d-1]q}{d-1}=\frac{44}{3}\ ,$$ 
i.e. $|\overline{\sigma}(l)|\leq 14$. Note that 
$|\overline{\mathcal{B}}(l)|=N_4(\Pi\setminus L)=16$. Put $r:=|\overline{\mathcal{B}}(l)\setminus
\overline{\Omega}(l)|, s:=|\overline{\mathcal{B}}(l)\setminus \overline{\sigma}(l)|$ with $s\leq r$ and observe that 
$$-r=-q^2+|\overline{\Omega}(l)|\leq -16+12=-4\ ,$$ 
$$-s=-q^2+|\overline{\sigma}(l)|\leq -16+14=-2\ .$$ Thus we have
$$N_q(X^3)\leq [(d-1)q+2]q+1+|\overline{\Omega}(l)|(d-1)q+$$
$$+|\overline{\sigma}(l)\setminus\overline{\Omega}(l)|\cdot [(d-1)q-(d-2)]+|\overline{\mathcal{B}}(l)\setminus\overline{\sigma}(l)|\cdot [(d-2)q+1]=$$
$$=57+12(16-r)+10(r-s)+9s=249-2r-s\leq 239\ ,$$
that is, $N_q(X^3)<\theta_3^{4,4}=245$.
\end{proof}

\medskip

The next example, 
the nonsingular parabolic quadric hypersurface in $\mathbb{P}^4$ and the nonsingular Hermitian hypersurface in $\mathbb{P}^4$, 
show that the equality in Proposition \ref{n=3 cone} is reached at least for the cases $d=2,\sqrt{q}+1$ and $q+1$.

\begin{ex}\label{example}
Consider the following polynomial $F\in\mathbb{F}_2[x_0,\dots,x_4]:$
$$F:= x_0^2x_1 + x_0x_1^2 + x_0x_1x_2 + 
x_0x_2^2 + x_1x_2^2 + x_2^3 + x_0x_2x_3 + 
x_2^2x_3 +$$ 
$$+ x_0^2x_4 + x_0x_1x_4 + x_1x_2x_4 + 
x_0x_3x_4 + x_2x_3x_4 + x_3^2x_4 + x_2x_4^2 + x_3x_4^2\ .$$
Then $X^3=V(F)\subset\mathbb{P}^4$ is a nonsingular hypersurface of degree $3$ over $\mathbb{F}_2$ with $N_q(X^3)=27=\theta_3^{3,2}$.
\end{ex}

\smallskip
 
\subsection{Cases $d\leq q$.}\label{sub1} By Proposition \ref{n=3 cone} we can assume, without loss of generality, that
$S_x:=T_xX\cap X^3$ is not a cone for any $\mathbb{F}_q$-point $x\in X^3$. Moreover, from \cite[Proposition $5$]{T1} we deduce that there exists at least an 
$\mathbb{F}_q$-line $l$ in $X^3$. Let $P\in l$ be an $\mathbb{F}_q$-point and note that $l\subset S_x$ for any $\mathbb{F}_q$-point $x\in l$. 

Denote by $\delta_x$ the number of $\mathbb{F}_q$-lines contained in $X^3$ and passing through $x\in l(\mathbb{F}_q)$. As in \cite{D}, define $\Omega (l):=\bigcup_{Q\in l(\mathbb{F}_q)}\Omega_l(Q)$ and let $\sigma(l)$ be the set of all $\mathbb{F}_q$-planes containing $l$ and such that their intersection with $X^3$ is a union of $d$ $\mathbb{F}_q$-lines. 

\medskip

\noindent\textit{Claim $2$}. If $d\leq q$, then we have $|\Omega (l)|\leq q^2-1$ and $|\sigma (l)|\leq q^2+q-2$.

\smallskip

\noindent\textit{Proof.} Since $S_x$ is not a cone for every $\mathbb{F}_q$-point $x\in X$, from Lemma \ref{lem-irred} and \cite[Theorem 3.1]{D} we deduce that $\delta_x\leq (d-1)d\leq (d-1)q$ and hence
$$|\Omega_l(x)|\leq \frac{(d-1)q-1}{d-1}=q-\frac{1}{d-1}<q\ ,$$ i.e.
$|\Omega_l(x)|\leq q-1$ for every $x\in l(\mathbb{F}_q)$.
Then we obtain that $$|\Omega (l)|\leq (q-1)(q+1)=q^2-1$$ and 
$$|\sigma (l)|\leq \frac{[(d-1)q-1](q+1)}{d-1}=
\frac{(d-1)q(q+1)-(q+1)}{d-1}=$$
$$=q(q+1)-\frac{q+1}{d-1}<q(q+1)-1\ ,$$
i.e. $|\sigma (l)|\leq q(q+1)-2=q^2+q-2$. \hfill{Q.E.D.}

\medskip

Note that
$$N_q(X^3)\leq N_q(l)+|\Omega (l)|\cdot (d-1)q+|\sigma (l)\setminus\Omega (l)|\cdot [(d-1)q-(d-2)]+$$
$$+[q^2+q+1-|\sigma (l)|]\cdot [(d-2)q+1]\ .$$
Setting $r:=q^2+q+1-|\Omega (l)|$ and $s:=q^2+q+1-|\sigma (l)|$, by \textit{Claim $2$} we see that $r\geq q+2$, $s\geq 3$ and this gives
$$N_q(X^3)\leq (q+1)+(q^2+q+1-r)(d-1)q+(r-s)[(d-1)q-(d-2)]+$$
$$+s[(d-2)q+1]=\theta_3^{d,q}+(1-r)(d-1)q+(r-s)[(d-1)q-(d-2)]+$$
$$+s[(d-2)q+1]=\theta_3^{d,q}+(d-1)q-r(d-2)-s(q-d+1)\leq \theta_3^{d,q}+$$
$$+(d-1)q-(q+2)(d-2)-3(q-d+1)=\theta_3^{d,q}+dq-q-dq+$$
$$+2q-2d+4-3q+3d-3=\theta_3^{d,q}-2q+d+1\leq \theta_3^{d,q}-q+1
<\theta_3^{d,q}\ .$$

\medskip

These facts, together with Proposition \ref{n=3 cone}, prove Theorem \ref{thm1} 
for $d\leq q$.

\subsection{Cases $d\geq q+1$.}\label{sub2} If $d\geq q+2$, then $N_q(X^3)<\theta_{3}^{d,q}$. So, in this subsection, from now on we can assume that $d=q+1$.  

\smallskip

Let $P\in X^{3}(\mathbb{F}_q)$. After a projectivity, we can assume that 
$$P=(1:0:0:0:0),\quad T_PX^{3}:=V(x_{4})$$ and
$$S_P:=T_PX^{3}\cap X^3=V(x_{4}, x_0^{m}F_{q+1-m}+ \dots + x_0F_{q}+F_{q+1})\ ,$$ where $m\leq q-1$ is the maximum positive integer such that $F_{q+1-m}\neq 0$ and $F_i\in\mathbb{F}_q[x_1,x_2,x_3]$ are homogeneous polynomials of degree $i$ such that $F_{q+1}\neq 0$ because $X^3$ does not contain any $\mathbb{F}_q$-plane, and $(F_{q+1-m},\dots,F_q)\neq (0,\dots,0)$ by Proposition \ref{n=3 cone}. Define $$W_P':=V(x_0,x_{4},F_{q+1-m},\dots,F_{q+1})\subseteq V(x_0,x_4,F_{q+1-m})$$ and observe that $W_P'(\mathbb{F}_q)$ is a set parametrizing all the $\mathbb{F}_q$-lines contained in $S_P$ and passing through $P$.

\medskip

\noindent\textit{Claim $3$}. If $q\geq 3$, then we have either 

\begin{enumerate}

\item[(I)] $N_q(W_P')\leq q^2$, or 

\smallskip

\item[(II)] $m=1$, $S_P=V(x_4,x_0F_q+F_{q+1})$, $V(x_0,x_4,F_{q})$ is a union of $\mathbb{F}_q$-lines $L_1,\dots, L_q\subset V(x_0,x_4)=:\mathbb{P}_P^2$ such that $\bigcap_{1=1}^{q}L_i$ is an $\mathbb{F}_q$-point $Q$ and $N_q(W_P')=N_q(V(x_0,x_4,F_{q}))=q^2+1$. 

\end{enumerate}

\medskip

\noindent\textit{Proof.} If $2\leq m\leq q-1$, then $N_q(W_P')\leq (q+1-m)q+1\leq q^2+1-q\leq q^2-2$. Hence, suppose that 
$$m=1\ \mathrm{and\ } S_P=V(x_4,x_0F_q+F_{q+1})\ .$$ If $V(x_0,x_4,F_{q})$ is not a union of $\mathbb{F}_q$-lines, then by \cite[Corollary 1.7]{Co} we get $N_q(W_P')\leq N_q(V(x_0,x_4,F_{q}))\leq (q-1)q+2\leq q^2-1$. Thus, assume that $V(x_0,x_4,F_{q})$ is a union of $\mathbb{F}_q$-lines $L_1,\dots, L_q$.
If $\bigcap_{1=1}^{q}L_i$ is not a single $\mathbb{F}_q$-point, then $N_q(W_P')\leq N_q(V(x_0,x_4,F_{q}))\leq (q-1)q+1+(q+1)-(q-1)=q^2+3-q\leq q^2$, while if $\bigcap_{1=1}^{q}L_i$ is an $\mathbb{F}_q$-point $Q$ then $N_q(V(x_0,x_4,F_{q}))=q^2+1$.
Therefore, we conclude that either $N_q(W_P')\leq q^2$, or $N_q(W_P')=N_q(V(x_0,x_4,F_{q}))=q^2+1$. \hfill{Q.E.D.} 

\bigskip

\noindent Note that from the proof of \cite[Proposition 5]{T1} and \cite{D} it follows that either $N_q(X^3)<\theta_3^{d,q}$, or through every $\mathbb{F}_q$-point $P\in X^3$ there exists at least an $\mathbb{F}_q$-line $l$ such that $P\in l$.

\medskip

First of all, assume that $q\geq 3$. 

\medskip

\noindent\textit{Claim $4$}. There exists an $\mathbb{F}_q$-line $l\subset X^3$ such that $|\Omega(l)|\leq q^2-1$.

\medskip

\noindent\textit{Proof.} Consider an $\mathbb{F}_q$-line $l\subset X^3$. If for every $P\in l(\mathbb{F}_q)$ we are in case (I) of \textit{Claim $3$}, then $|\Omega_l(P)|\leq\frac{q^2-1}{q}=q-\frac{1}{q}<q$, i.e. $|\Omega_l(P)|\leq q-1$ for every $P\in l(\mathbb{F}_q)$ and $|\Omega(l)|\leq \sum_{P\in l(\mathbb{F}_q)}|\Omega_l(P)|\leq (q-1)(q+1)=q^2-1$.

Thus, suppose we are in case (II) of \textit{Claim $3$} for some $P\in l(\mathbb{F}_q)$. Consider the $\mathbb{F}_q$-line $l':=\langle P,O\rangle$, for some $O\in L_1(\mathbb{F}_q)$ and $O\neq Q$. Then $|\Omega_{l'}(P)|=1$. If for every $\mathbb{F}_q$-point $P'\in l'\setminus\{P\}$ we are in case (I) of \textit{Claim $3$}, then $|\Omega_{l'}(P')|\leq q-1$ and so $|\Omega(l')|\leq 1+(q-1)q < q^2-1$. 

Assume now that 
we are in case (II) of \textit{Claim $3$} for some $P'\in l'(\mathbb{F}_q)\setminus\{P\}$. After a suitable projectivity, using a similar notation as in \textit{Claim $3$}, we get that $S_{P'}=V(x_4,x_0F'_q+F'_{q+1})$, $V(x_0,x_4,F'_{q})$ is a union of $\mathbb{F}_q$-lines $L'_1,\dots, L'_q\subset V(x_0,x_4)=:\mathbb{P}_{P'}^2$ such that $Q'=\bigcap_{1=1}^{q}L'_i$ and $N_q(W_{P'}')=N_q(V(x_0,x_4,F'_{q}))=q^2+1$.  
Then, we have the following two possibilities: 

\smallskip

(j) $l'=\langle P', P''\rangle$ for some $P''\in L'_1(\mathbb{F}_q)\setminus\{Q'\}$; or 

\smallskip

(jj) $l'=\langle P',Q'\rangle$. 

\smallskip

\noindent In case (j), we obtain that $|\Omega_{l'}(P')|=1$ and then 
$$|\Omega (l')|\leq \sum_{A\in l'(\mathbb{F}_q)\setminus\{P,P'\}}|\Omega_{l'}(A)| +1+1\leq (q-1)q+2\leq q^2-1\ .$$ 

\noindent Suppose we are in case (jj). Observe that the $\mathbb{F}_q$-plane $\langle P',L'_1\rangle$ is contained in $T_{P'}X^3$ but not in $T_{P}X^3$ because $P\in\langle P',L'_1\rangle$ and $\langle P',L'_1\rangle\cap X^3$ is not singular at $P$. Hence $T_PX^3\neq T_{P'}X^3$. Define $T_PX^3\cap T_{P'}X^3=:\pi'=\mathbb{P}^2$. Since $l'\subset\pi'\subset T_{P'}X^3$, we deduce that $\pi'$ intersects $\mathbb{P}_{P'}^2$ in the unique $\mathbb{F}_q$-line $L'\subset \mathbb{P}_{P'}^2$ containing $Q'$ and distinct from all the $L'_j$'s. Moreover, since $l'\subset\pi'\subset T_{P}X^3$, we see that in $\pi'$ there are at least $q-1\geq 2$ $\mathbb{F}_q$-lines different from $l'$ and passing through $P\neq P'$. This shows that there exists an $\mathbb{F}_q$-point $Q''\in T_{P'}X^3\cap X^3$ such that $Q''\notin \bigcup_{j=1,...,q}\langle P',L'_j\rangle$. Consider an $\mathbb{F}_q$-line $L\subset T_{P'}X^3$ such that $Q''\in L$ and $l'\cap L=\emptyset$. Taking the $\mathbb{F}_q$-plane $\pi:=\langle P',L\rangle$, it follows that $\pi\cap X^3=\bigcup_{j=1}^{q+1}l_j$ is a union of distinct $\mathbb{F}_q$-lines $l_j$ such that, after renaming, $\bigcap_{i=1}^{q}l_i=P'$ and $P'\notin l_{q+1}$. Define $l:= l_{q+1}$ and set $l\cap \bigcup_{i=1}^{q}l_i=
\{P_1,\dots,P_q\}$ and $l(\mathbb{F}_q)=\{Q''=P_0,P_1,\dots,P_q\}$. 
For any $j=1,\dots,q$, after a projectivity, from \textit{Claim $3$} it follows that 
either (k) $N_q(W_{P_j}')\leq q^2$, or (kk) $S_{P_j}:=T_{P_j}X^3\cap X^3=V(x_4,x_0F_q^{(j)}+F_{q+1}^{(j)})$, $V(x_0,x_4,F_{q}^{(j)})$ is a union of $\mathbb{F}_q$-lines $L_1^{(j)},\dots, L_q^{(j)}\subset \mathbb{P}_{P_j}^2$ such that $\bigcap_{j=1}^{q}L_i^{(j)}$ is an $\mathbb{F}_q$-point $Q_j$ and $N_q(W_{P_j}')=N_q(V(x_0,x_4,F_{q}^{(j)}))=q^2+1$. 

Observe that case (kk) cannot occur because $\pi\subset T_{P_j}X^3$ and the $\mathbb{F}_q$-plane $\pi$ intersects $V(x_0,x_4,F_{q}^{(j)})=\bigcup_{i=1}^{q}L_i^{(j)}$ at $1$, $q+1$ or $q>2$ $\mathbb{F}_q$-points, giving a contradiction because in $\pi$ there are exactly two distinct $\mathbb{F}_q$-lines through each $P_j$ for every $j=1,\dots,q$. Thus we have 
$$N_q(W_{P_j}')\leq q^2\qquad \forall j=1,\dots,q\ ,$$
and this gives $|\Omega_l(P_j)|\leq \frac{q^2-1}{q}=q-\frac{1}{q}<q$, i.e. 
$$|\Omega_l(P_j)|\leq q-1 \qquad \forall j=1,\dots,q\ .$$ 

As to the $\mathbb{F}_q$-point $P_0$, note that $P_0\in l\subset T_{P_0}X^3$ and $\pi\not\subset T_{P_0}X^3$ because $\pi\cap X^3$ is not singular at $P_0$. Then, by \textit{Claim $3$} we get that either 

\begin{enumerate}
\item[(a)] $N_q(W_{P_0}')\leq q^2$ or, after a projectivity, 
\item[(b)] $S_{P_0}:= T_{P_0}X^3\cap X^3=V(x_4,x_0F_q^{(0)}+F_{q+1}^{(0)})$, $V(x_0,x_4,F_{q}^{(0)})$ is a union of $\mathbb{F}_q$-lines $L_1^{(0)},\dots, L_q^{(0)}\subset\mathbb{P}_{P_0}^2$ such that $\bigcap_{1=1}^{q}L_i^{(0)}$ is an $\mathbb{F}_q$-point $Q_0$ and $N_q(W_{P_0}')=N_q(V(x_0,x_4,F_{q}^{(0)}))=q^2+1$. 
\end{enumerate}

\smallskip

In case (a), we have 
$|\Omega_l(P_0)|\leq\frac{q^2-1}{q}<q$, i.e. $|\Omega_l(P_0)|\leq q-1$, and this gives
$|\Omega(l)|\leq \sum_{j=0}^{q}|\Omega_l(P_j)|\leq (q-1)(q+1)=q^2-1$. 

\smallskip

If we are in case (b), then either $Q_0\notin l$, $|\Omega_l(P_0)|\leq 1$ and $|\Omega(l)|\leq |\Omega_l(P_0)|+\sum_{j=1}^{q}|\Omega_l(P_j)|\leq 1+q(q-1) < q^2-1$, or $Q_0\in l$. 

Assume that $Q_0\in l$.
Since $\pi\not\subset T_{P_0}X^3$ and $\pi\subset T_{P_j}X^3$ for every $j=1,\dots,q$, we see that $\pi_j':=T_{P_0}X^3\cap T_{P_j}X^3$ is an $\mathbb{F}_q$-plane containing the $\mathbb{F}_q$-line $l=\langle P_0,Q_0\rangle$ and such that $\pi_j'$ does not contain any $\mathbb{F}_q$-line $L_i^{(0)}$, otherwise $\pi_j'\cap X^3\subseteq T_{P_j}X^3\cap X^3$ would be not singular at $P_j$. Thus 
$\pi_j'\subseteq T_{P_0}X^3$ intersects $\mathbb{P}_{P_0}^2\subseteq T_{P_0}X^3$ in the unique $\mathbb{F}_q$-line $L'_0$ distinct from the $L_i^{(0)}$'s and such that $Q_0\in L'_0$. This shows that 
$$\langle l,L'_0\rangle =\pi_j'\subset T_{P_j}X^3 $$ 
for every $j=1,\dots,q$. Since $\pi\subset T_{P_j}X^3$, $\langle l,L'_0\rangle\subset T_{P_j}X^3$ and $\pi\neq \langle l,L'_0\rangle\subseteq T_{P_0}X^3$, we deduce that
$\pi\ \cup \langle l,L'_0\rangle \subseteq T_{P_j}X^3$, that is, $T_{P_j}X^3=T=\mathbb{P}^3$ for every $j=1,\dots,q$.
Let $L''$ be an $\mathbb{F}_q$-line contained in $\mathbb{P}_{P_0}^2\subset T_{P_0}X^3$ and not passing through $Q_0$.
Consider now the $\mathbb{F}_q$-plane $\langle L'',P'\rangle$ and note that $l\ \cap \langle L'',P'\rangle  = \emptyset$, otherwise $P'\in T_{P_0}X^3$ and $\pi =\langle P',l\rangle\subseteq T_{P_0}X^3$, a contradiction. Moreover, observe that $T_{P_j}X^3\cap \langle L'',P'\rangle = T\cap \langle L'',P'\rangle $ is an $\mathbb{F}_q$-line $\hat{L}$ for every $j=1,\dots, q$, otherwise $l\cup\langle L'',P'\rangle\subseteq T_{P_j}X^3=T=\mathbb{P}^3$ and then $l \cap \langle L'',P'\rangle \neq \emptyset$, a contradiction. Thus $L''$ and $\hat{L}$ intersect in a point and they parametrize all the $\mathbb{F}_q$-planes containing $l$ and contained in $T_{P_0}X^3$ and $T=T_{P_j}X^3$, respectively.
This gives $|\Omega(l)|\leq 2q+1<q^2-1$ for any $q\geq 3$. \hfill{Q.E.D.} 

\bigskip

From \textit{Claim $4$} it follows that there is an $\mathbb{F}_q$-line $l\subset X^3$ such that $r:=q^2+q+1-|\Omega (l)|\geq q+2$. By arguing as in Subsection  \ref{sub1} and using the same notation as there, since $d=q+1$ we have
$$N_q(X^3)\leq  \theta_3^{d,q}+(d-1)q-r(d-2)-s(q-d+1) =$$
$$= \theta_3^{q+1,q}+q^2-r(q-1)\leq \theta_3^{q+1,q}+q^2-(q+2)(q-1)= $$
$$=\theta_3^{q+1,q}+2-q<\theta_3^{q+1,q}\ .$$

\medskip

Finally, assume that $q=2$. Hence $d=3$ and consider an $\mathbb{F}_2$-point $P\in X^3$ and an $\mathbb{F}_2$-line $l$ containing $P$. Then $l\subset T_PX^3$ and, after a projectivity, we have $S_P=T_PX^3\cap X^3=V(x_4,x_0F_2+F_3)$, where $F_i\in\mathbb{F}_2[x_1,x_2,x_3]$ for $i=2,3$, and one of the following possibilities can occur (e.g., see \cite[Theorem 5.2.4, $\S 7.2$]{H}):
\begin{enumerate}
\item[(A)] $F_2=x_1^2+x_2x_3$;
\item[(B)] $F_2=x_1^2+x_1x_2+x_2^2$;
\item[(C)] $F_2=x_1^2$;
\item[(D)] $F_2=x_1x_2$.
\end{enumerate}    
Write $\Gamma:=V(x_0,x_4,F_2)$. In Case (B), since $N_2(\Gamma)=1$, we have $l=\langle P,P'\rangle$, where $P'\in\Gamma(\mathbb{F}_2)$. Therefore, for any $\mathbb{F}_2$-plane $\pi\subset T_PX^3$ containing $l$ we deduce that $\pi\cap X^3=l\cup\gamma$. If $\gamma$ is irreducible, then $P\in\gamma$ and so $N_2((\pi\cap X^3)\setminus \{l\})\leq q=2$. On the other hand, if $\gamma$ is the union of two $\mathbb{F}_2$-lines $L_1,L_2$, then after renaming we see that $L_1=l$ and thus $N_2((\pi\cap X^3)\setminus \{l\})\leq q=2$. Hence we have $N_2((\pi\cap X^3)\setminus \{l\})\leq q=2$ for every $\mathbb{F}_2$-plane $\pi\subset T_PX^3$ such that $l\subset\pi$. So, we conclude that $N_2(S_P)=N_2(T_PX^3\cap X^3)\leq (q+1)q+(q+1)=9<11=(d-1)q^2+q+1,$ that is $N_2(X^3)<\theta_3^{3,2}$.  
In case (D), we see that $$X^3=V(x_4G_2+x_0x_1x_2+F_3)\ ,$$ where $G_2\in\mathbb{F}_2[x_0,\dots,x_4]$ with $\deg G_2=2$ and $F_3\in\mathbb{F}_2[x_1,x_2,x_3]$. Using the following Magma program 

\bigskip

\begin{verbatim}
K:=GF(2);
P<[x]>:=ProjectiveSpace(K,4);
AA:=VectorSpace(K,15);
DD:=VectorSpace(K,10);
A:=[x : x in AA];
D:=[x : x in DD];
i:=0;
List:={};
for a in [1..#A] do
 for d in [1..#D] do
 i:=i+1;
 AA:=A[a][1]*x[1]^2+A[a][2]*x[2]^2+A[a][3]*x[3]^2
 +A[a][4]*x[4]^2+A[a][5]*x[5]^2+A[a][6]*x[1]*x[2]
 +A[a][7]*x[1]*x[3]+A[a][8]*x[1]*x[4]+A[a][9]*x[1]*x[5]
 +A[a][10]*x[2]*x[3]+A[a][11]*x[2]*x[4]+A[a][12]*x[2]*x[5]
 +A[3][13]*x[3]*x[4]+A[a][14]*x[3]*x[5]+A[a][15]*x[4]*x[5];
 DD:=D[d][1]*x[2]^3+D[d][2]*x[3]^3+D[d][3]*x[4]^3
 +D[d][4]*x[2]^2*x[3]+D[d][5]*x[2]^2*x[4]+D[d][6]*x[3]^2*x[4]
 +D[d][7]*x[2]*x[3]^2+D[d][8]*x[2]*x[4]^2+D[d][9]*x[3]*x[4]^2
 +D[d][10]*x[2]*x[3]*x[4];
 X:=Scheme(P,[x[5]*(AA)+x[1]*x[2]*x[3]+(DD)]);
  if IsNonSingular(X) and #Points(X) ge 27 then 
  PP:=[x : x in Points(X)];
  h:=0;
   for j in [1..#PP] do
   TS:=TangentSpace(X,PP[j]);
   S:=TS meet X;
   TC:=TangentCone(S,PP[j]);
    if TC eq S then
    h:=h+1;
    end if;
   end for;
   if h eq 0 then
   X;
   #Points(X);
   x[5]*(AA)+x[1]*x[2]*x[3]+(DD);
   List:=List join {x[5]*(AA)+x[1]*x[2]*x[3]+(DD)};
   end if;
  end if;
  i;
 end for;
end for;
\end{verbatim}

\bigskip

\noindent we can see that there are no
nonsingular hypersurfaces $X^3\subset\mathbb{P}^4$ with $N_q(X^3)\geq 27$ which satisfy 
the condition that $T_PX^3\cap X^3$ is not a cone for every $\mathbb{F}_q$-point $P\in X^3$, that is, $N_2(X^3)<\theta_3^{3,2}$.  

\medskip

\noindent Therefore, for every $P\in X^3(\mathbb{F}_2)$ we can assume that $S_P:=T_pX^3\cap X^3$ is as in cases (A), or (C).

In case (A), we see that $\Gamma$ is a nonsingular conic. Set $l=\langle P,P'\rangle$ for some $P'\in\Gamma (\mathbb{F}_2)$. Consider the tangent $\mathbb{F}_2$-line $L_1:=T_{P'}\Gamma$ and the $\mathbb{F}_2$-plane $\pi_1:=\langle P,L_1\rangle$. Then $\pi_1\cap X^3=l\cup \gamma$, where $\gamma$ is a conic. If $\gamma$ is irreducible, then $P\in \gamma$ and so $N_2(\pi_1\cap X^3\setminus\{l\})\leq q=2$.
If $\gamma=l'\cup l''$ is the union of two $\mathbb{F}_2$-lines $l',l''$, then after renaming we obtain that $P\in l'=l$ and then $N_2(\pi_1\cap X^3\setminus\{l\})\leq q=2$. Consider now the other two $\mathbb{F}_2$-lines $L_2,L_3$ contained in $\mathbb{P}_P^2:=V(x_0,x_4)$, distinct from $L_1$ and passing through $P'$. Set $\pi_i:=\langle P,L_i\rangle$ for $i=2,3$. Then $\pi_i\cap X^3=l\cup\gamma_i$ for $i=2,3$. If $\gamma_i$ is an irreducible conic, then $P\in\gamma_i$ and so $N_2(\pi_i\cap X^3\setminus\{l\})\leq q=2$. If $\gamma_i$ is the union of two $\mathbb{F}_2$-lines $l_i',l_i''$, then after renaming, we obtain that $P\in l_i'$ and $P\notin l_i''$ if $l_i''\neq l$. Thus we get $N_2(\pi_i\cap X^3\setminus\{l\})\leq 2q+1-2=q+1$ for $i=2,3$. This gives $N_2(S_P)\leq q+2(q+1)+(q+1)=11=(d-1)q^2+q+1$.

Suppose now we are in case (C). Then $\Gamma$ is a double $\mathbb{F}_2$-line and $l:=\langle P,P'\rangle$ for some $P'\in \Gamma(\mathbb{F}_2)$. Consider the two distinct $\mathbb{F}_2$-lines $L_2', L_3'$ in $\mathbb{P}_P^2$ passing through $P'$ and not contained in $\Gamma$. Define $\pi_i':=\langle P,L_i'\rangle$ for $i=2,3$. Then $\pi_i'\cap X^3=l\cup\gamma_i'$, where $\gamma_i'$ is a conic for $i=2,3$. Then by arguing as in case (A) for $\pi_1$, we deduce that $N_2(\pi_i'\cap X^3\setminus\{l\})\leq q=2$ for $i=2,3$. This shows that $N_2(S_P)\leq q+q+2q+(q+1)=5q+1=11=(d-1)q^2+q+1$. 

On the other hand, if we are in case (A) for some $P\in X^3(\mathbb{F}_2)$, considering $\pi_1=\langle P,L_1\rangle$ again, we can see that there exists a nonsingular $\mathbb{F}_2$-point $P''$ of $\pi_1\cap X^3$ ($P''\in l\setminus\gamma$ when $\gamma$ is irreducible, or $P''\in l''\setminus l$ when $\gamma$ is reducible in two $\mathbb{F}_2$-lines) such that $\pi_1\not\subset T_{P''}X^3$ and there exists an $\mathbb{F}_2$-line $L\subset\pi_1$ such that $P''\in L\not\subset T_{P''}X^3$ and $N_2(X^3\cap L\setminus\{P''\})\leq 1$. Since $S_{P''}$ can be assumed to be as in case (A), or (C), we obtain that
$$N_2(X^3)\leq N_2(S_{P''})+(d-1)(q^3-1)+(d-2)\leq\theta^{3,2}_3-1<\theta^{3,2}_3\ .$$ 

Finally, if we are in case (C) for some $P\in X^3(\mathbb{F}_2)$, by considering for example the $\mathbb{F}_2$-plane $\pi_2'$ and arguing as above for the case (A), we can see that there exists a nonsingular $\mathbb{F}_2$-point $P''$ of $\pi_2'\cap X^3$ such that $\pi_2'\not\subset T_{P''}X^3$ and there exists an $\mathbb{F}_2$-line $L'\subset\pi_2'$ with $P''\in L'\not\subset T_{P''}X^3$ and such that $N_2(X^3\cap L\setminus\{P''\})\leq 1$. Since $S_{P''}$ can be assumed to be as in case (A), or (C), we conclude again that
$$N_2(X^3)\leq N_2(S_{P''})+(d-1)(q^3-1)+(d-2)\leq\theta^{3,2}_3-1<\theta^{3,2}_3\ .$$    

\smallskip

\noindent All the above arguments complete the proof of Theorem \ref{thm1} and the previous Magma program also provides further
examples of nonsingular hypersurfaces $X^3\subset\mathbb{P}^4$ of degree $3$ defined over $\mathbb{F}_2$
with $N_2(X^3)=\theta_{3}^{3,2}$.

\section{Proof of Theorem \ref{thm2}}\label{sec 3}

\noindent Write $n:=2k+1$, where $k\in\mathbb{Z}_{\geq 1}$, and suppose that $d\leq q$.

\smallskip

Assuming that Theorem \ref{thm2} is true for some $k\in\mathbb{Z}_{\geq 1}$, our main goal here will be to prove it for $k+1$.
As a consequence of this inductive argument, since from Subsection \ref{sub1} we know that Theorem \ref{thm1} holds for $d\leq q$, it will follow that Theorem \ref{thm2} is true for any odd integer $n\geq 5$.  

\smallskip

So, let us assume that Theorem \ref{thm2} holds for some $k\in\mathbb{Z}_{\geq 1}$ and consider a nonsingular hypersurface $X^{2k+3}$ of degree $d\geq 2$ in $\mathbb{P}^{2(k+2)}$ defined over $\mathbb{F}_q$ with $d\leq q$. 

\medskip

First of all, let us prove the following preliminary result.    

\bigskip

\noindent\textit{Claim $5$}. If there exists an $\mathbb{F}_q$-point $P\in X^{2k+3}$ such that 
$T_PX^{2k+3}\cap X^{2k+3}$ is a cone $P*W^{2k+1}$ over some hypersurface $W^{2k+1}\subset\mathbb{P}^{2(k+1)}$, then $W^{2k+1}$ is nonsingular, $N_q(X^{2k+3})\leq\theta_{2k+3}^{d,q}$
and the equality is attained only if  $N_q(W^{2k+1})=\theta_{2k+1}^{d,q}$.

\medskip

\noindent\textit{Proof.} By Lemma \ref{lem-cone}, we know that $W^{2k+1}$ is a nonsingular hypersurface defined over $\mathbb{F}_q$. Thus, by the inductive hypothesis, we have $N_q(W^{2k+1})\leq \theta_{2k+1}^{d,q}$. Hence, we conclude that $$N_q(X^{2k+3})\leq (d-1)q^{2k+3}+N_q(P*W^{2k+1})\leq (d-1)q^{2k+3}+$$
$$+qN_q(W^{2k+1})+1\leq (d-1)q^{2k+3}+q\ \theta_{2k+1}^{d,q}+1=:\theta_{2k+3}^{d,q}\ ,$$ and if $N_q(X^{2k+3})=\theta_{2k+3}^{d,q}$, then necessarily $N_q(W^{2k+1})=\theta_{2k+1}^{d,q}$. \hfill{Q.E.D.}

\bigskip

\noindent By \textit{Claim} $5$ and Lemma \ref{lem-irred}, we can assume that $Y_Q:=T_QX^{2k+3}\cap X^{2k+3}$ is irreducible over $\mathbb{F}_q$ and it is not a cone for every $Q\in X^{2k+3}(\mathbb{F}_q)$. 

\medskip

\noindent Let $P\in X^{2k+3}(\mathbb{F}_q)$. As in Subsection \ref{sub2}, after a suitable projectivity, we can suppose that 
$$P=(1:0:\dots :0),\quad T_PX^{2k+3}:=V(x_{2(k+2)})$$ and
$$Y_P:=V(x_{2(k+2)}, x_0^{d-2}F_2+ \dots + x_0F_{d-1}+F_d)\ ,$$ where $F_i\in\mathbb{F}_q[x_1,\dots,x_{2k+3}]$ are homogeneous polynomials of degree $i$ for every $i=2,\dots, d$; moreover, note that $(F_2,\dots ,F_{d-1})\neq (0,\dots,0)$ and $F_d\neq 0$ by the hypothesis on $Y_P$. Let $j$ be the minimum index such that $F_j\neq 0$ for $j=2,\dots,d-1$. 
Define $$W_P:=V(x_0,x_{2(k+2)},F_j,\dots,F_d)\subseteq X^{2k+3}$$ and note that $W_P(\mathbb{F}_q)$ is the set parametrizing all the $\mathbb{F}_q$-lines in $Y_P$ passing through $P$, i.e. $l$ is an $\mathbb{F}_q$-line contained in $Y_P\subseteq X^{2k+3}$ and passing through $P$ if and only if $l$ is an $\mathbb{F}_q$-line passing through $P$ and an $\mathbb{F}_q$-point $Q\in W_P$. 
Note that $W_P\subset V(x_0,x_{2(k+2)})=\mathbb{P}^{2k+2}$. 

\bigskip

\noindent\textit{Claim $6$}. If $d\leq q$, then we have
$N_q(W_P)\leq (d-1)q^{2k+1}+N_q(\mathbb{P}^{k-1})$ for any $P\in X^{2k+3}(\mathbb{F}_q)$.

\medskip

\noindent\textit{Proof.} If $j\leq d-2$, then 
we have
$$N_q(W_P)\leq N_q(V(x_0,x_{2(k+2)},F_j))\leq j\ q^{2k+1}+q^{2k}+\dots+q+1\leq$$
$$\leq (d-1)q^{2k+1}-q^{2k+1}+(q^{2k}+\dots+1)<(d-1)q^{2k+1}\ .$$
   
\noindent If $j=d-1$, then $Y_P=V(x_{2(k+2)}, x_0F_{d-1}+F_d)$. Define 
$$W_{d-1}^{2k+1}:=V(x_0,x_{2(k+2)},F_{d-1})\subset \mathbb{P}^{2(k+1)}=V(x_0,x_{2(k+2)})\ $$
and observe that $N_q(W_P)\leq N_q(W_{d-1}^{2k+1})$. Then we have the following two possibilities: 

\medskip

$(\alpha)$ $W_{d-1}^{2k+1}$ is not a union of $\mathbb{P}^{2k+1}$'s; 

\smallskip

$(\beta)$ $W_{d-1}^{2k+1}=\bigcup_{i=1}^{d-1}L_i$, where each $L_i$ is an $\mathbb{F}_q$-linear subspace $\mathbb{P}^{2k+1}$. 

\medskip

\noindent In case $(\alpha)$, write $W_{d-1}^{2k+1}:=\left(\bigcup_{j=1}^{s}L_j\right)\cup W'$ for $0\leq s\leq d-3$ and $3\leq d\leq q$, where $L_j=\mathbb{P}^{2k+1}$ for any $j=1,\dots,s$ and $W'$ is a hypersurface in $\mathbb{P}^{2(k+1)}=V(x_0,x_{2(k+2)})$ defined over $\mathbb{F}_q$ of degree $d-s-1$ and without $\mathbb{F}_q$-linear components. If $s=0$, then
from \cite{HK4} it follows that
$$N_q(W_{d-1}^{2k+1})\leq (d-1)q^{2k+1}-q^{2k+1}+dq^{2k}-q^{2k}+q^{2k-1}+\dots+q+1\leq $$
$$\leq (d-1)q^{2k+1}-(q^{2k}-q^{2k-1}-\dots-q-1)<(d-1)q^{2k+1}\ .$$
If $s\geq 1$, then $4\leq d\leq q$ and 
$$N_q(W_{d-1}^{2k+1})\leq (sq^{2k+1}+q^{2k}+q^{2k-1}+\dots+1)+(d-s-2)q^{2k+1}+$$
$$+(d-s-1)q^{2k}+q^{2k-1}+\dots+1\leq (d-1)q^{2k+1}-q^{2k+1}+(d-s)q^{2k}+$$
$$+2q^{2k-1}+\dots+2q+2\leq (d-1)q^{2k+1}-q^{2k}+2q^{2k-1}+\dots +2q+2=$$
$$= (d-1)q^{2k+1}-(q^{2k}-2q^{2k-1}-\dots -2q-2)<(d-1)q^{2k+1}\ .$$
Hence $N_q(W_P)<(d-1)q^{2k+1}$. 

Finally, consider case $(\beta)$. 
Define $$W_i:=L_i\cap V(x_0,x_{2(k+2)},F_d)\subseteq\mathbb{P}^{2k+1}$$ and observe that $k_i:=k_{W_i}\leq k$
for every $i=1,\dots,d-1$, because $W_i\subseteq W_P\subseteq X^{2k+3}$ and $k_{W_i}+1\leq k_{X^{2k+3}}\leq k+1$. 

Up to renaming, assume that $k\geq k_1\geq k_2\geq\dots\geq k_{d-1}$. Since from \cite[Theorem 3.2]{HK7} and \cite{T1} it follows that 
$$N_q(W_i)\leq (d-1)q^{k_i}N_q(\mathbb{P}^{2k-k_i})+N_q(\mathbb{P}^{k_i})=$$
$$= (d-1)q^{k_1}(q^{2k-k_1}+\dots+q)+dq^{k_1}+N_q(\mathbb{P}^{k_1-1})$$
for any $i=1,\dots,d-1$, we get
$$N_q(W_P)\leq N_q(W_1)+\sum_{j=2}^{d-1}\left[N_q(W_j)-N_q(\mathbb{P}^{k_1-1})\right]\leq $$
$$\leq (d-1)[(d-1)q^{2k}+(d-1)q^{2k-1}+\dots+(d-1)q^{k_1+1}+dq^{k_1}]+N_q(\mathbb{P}^{k_1-1})\leq$$
$$\leq (d-1)dq^{2k}+N_q(\mathbb{P}^{k_1-1})\leq
(d-1)q^{2k+1}+N_q(\mathbb{P}^{k-1})\ ,$$
which gives the statement. \hfill{Q.E.D.}

\bigskip

\noindent Observe that we can futher assume that 
$Y_P$ contains at least an $\mathbb{F}_q$-line $l$ passing through $P$, otherwise from the proof of \cite[Proposition $5$]{T1} we deduce that 
$$N_q(X^{2k+3})\leq (d-1)q^{2k+3}+(d-2)N_q(\mathbb{P}^{2k+2})+1=$$
$$=(d-1)\left(q^{2k+3}+q^{2k+2}+\dots+q^{k+2}+N_q(\mathbb{P}^{k+1})\right)-N_q(\mathbb{P}^{2k+2})+1 =$$
$$= (d-1)q^{k+2}N_q(\mathbb{P}^{k+1})+(d-1)N_q(\mathbb{P}^{k+1})-N_q(\mathbb{P}^{2k+2})+1 \leq$$ 
$$\leq (d-1)q^{k+2}N_q(\mathbb{P}^{k+1})+qN_q(\mathbb{P}^{k+1})+1-N_q(\mathbb{P}^{k+1})-N_q(\mathbb{P}^{2k+2})=$$
$$= (d-1)q^{k+2}N_q(\mathbb{P}^{k+1})+q^{k+2}-N_q(\mathbb{P}^{2k+2})=$$
$$= \theta_{2k+3}^{d,q}-N_q(\mathbb{P}^{k+1})+q^{k+2}-N_q(\mathbb{P}^{2k+2})<\theta_{2k+3}^{d,q}\ .$$ 
Let $\delta$ be the number of $\mathbb{F}_q$-planes containing $l$ and contained in $X^{2k+3}$, and let $\epsilon_P$ be the number of $\mathbb{F}_q$-lines passing through $P$ which are distinct from $l$ and not contained in the $\delta$ $\mathbb{F}_q$-planes. Since the $\mathbb{F}_q$-linear space $\mathbb{P}^{2(k+1)}=V(x_0,x_{2(k+2)})$ does not contain $l$, by \textit{Claim} $6$ we deduce that 
$$\delta q+\epsilon_P+1\leq N_q(W_P)\leq (d-1)q^{2k+1}+N_q(\mathbb{P}^{k-1}):=\tau\ ,$$
i.e. $\epsilon_P\leq \tau-\delta q-1$.
Since $P$ is any $\mathbb{F}_q$-point of $X^{2k+3}$, it follows that the total number of $\mathbb{F}_q$-lines distinct from $l$, not contained in the $\delta$ $\mathbb{F}_q$-planes and which intersect $l$, is given by $$\sum_{Q\in l(\mathbb{F}_q)}\epsilon_Q\leq (\tau-\delta q-1)(q+1)\ .$$ Therefore the number $\Omega$ of $\mathbb{F}_q$-planes containing $l$ and such that their intersection with $X^{2k+3}$ is a union of $d$ $\mathbb{F}_q$-lines is such that 
$$\Omega\leq \frac{(\tau-\delta q-1)(q+1)}{d-1}\leq \left[q^{2k+1}+\frac{N_q(\mathbb{P}^{k-1})}{d-1}-\delta-\frac{\delta+1}{d-1}\right](q+1)=:\sigma.$$ So, considering an $\mathbb{F}_q$-linear space $\Pi=\mathbb{P}^{2(k+1)}$ such that $l\cap \Pi=\emptyset$ and whose $\mathbb{F}_q$-points parametrize all the $\mathbb{F}_q$-planes containing $l$, we get
$$N_q(X^{2k+3})\leq N_q(l)+\delta q^2 + \Omega (d-1)q +\left(N_q(\Pi)-\delta-\Omega\right)[(d-2)q+1]=$$
$$=(q+1)+\delta q^2 + \Omega [(d-1)q-(d-2)q-1] +\left(N_q(\Pi)-\delta\right)[(d-2)q+1]\leq$$
$$\leq (q+1)+\delta q^2 + \sigma (d-1)q +\left(N_q(\Pi)-\delta-\sigma\right)[(d-2)q+1]=$$
$$=(q+1)+\delta q^2 + \sigma (d-1)q +\left(N_q(\Pi)-\delta-\sigma\right)[(d-1)q-(q-1)]=$$
$$=(q+1)+\delta q^2 +(d-1)q\left[N_q(\Pi)-\delta\right]-(q-1)\left[N_q(\Pi)-\delta-\sigma\right]\ .$$
Then we have
$$N_q(X^{2k+3})\leq (q+1)+\delta q^2 +(d-1)q\left[N_q(\Pi)-\delta\right]+$$
$$+(1-q)\left[N_q(\mathbb{P}^{2k})+\delta q-\frac{N_q(\mathbb{P}^{k-1})(q+1)}{d-1}+\frac{(\delta+1)(q+1)}{d-1}\right]=$$
$$=(q+1)+\delta q^2+\theta_{2k+3}^{d,q}-N_q(\mathbb{P}^{k+1})+(d-1)(q^{k+1}+\dots+q)+$$
$$+(1-d)q\delta -(q-1)N_q(\mathbb{P}^{2k})-\delta q (q-1)+\frac{N_q(\mathbb{P}^{k-1})(q^2-1)}{d-1}+$$
$$+\frac{(\delta+1)(q+1)(1-q)}{d-1}
\leq \theta_{2k+3}^{d,q}-N_q(\mathbb{P}^{k+1})+(q+1)+\delta q^2+$$
$$+(d-1)(q^{k+1}+\dots+q)-(d-1)q\delta -(q-1)N_q(\mathbb{P}^{2k})-\delta q (q-1)+$$
$$+N_q(\mathbb{P}^{k-1})(q^2-1)-(\delta+1)(q+1)\leq
\theta_{2k+3}^{d,q}-N_q(\mathbb{P}^{k+1})+$$
$$+(q-1)(q^{k+1}+\dots+q)-(d-1)q\delta -\delta+(1-q)N_q(\mathbb{P}^{2k})+$$
$$+N_q(\mathbb{P}^{k-1})(q^2-1)\leq
\theta_{2k+3}^{d,q}-2q+q^{k+2}-N_q(\mathbb{P}^{k-1})-q^{2k+1}\ ,$$
i.e. $N_q(X^{2k+3})<\theta_{2k+3}^{d,q}$. This concludes the proof of Theorem \ref{thm2}.

\end{document}